\documentclass[11pt,reqno]{amsart}

\usepackage[T1]{fontenc}
\usepackage{xcolor}
\usepackage{graphicx}
\usepackage{amsmath,amssymb,amsthm,mathtools}
\usepackage{aliascnt}
\usepackage{fullpage}
\usepackage[noadjust]{cite}
\usepackage[colorlinks=true,allcolors=blue,backref=page]{hyperref}
\usepackage[noabbrev,capitalize,nameinlink]{cleveref}

\allowdisplaybreaks[1]
\numberwithin{equation}{section}
\crefname{equation}{}{}

\theoremstyle{plain}
\newtheorem{theorem}{Theorem}[section]

\newaliascnt{proposition}{theorem}
\newtheorem{proposition}[proposition]{Proposition}
\aliascntresetthe{proposition}
\crefname{proposition}{Proposition}{Propositions}

\newaliascnt{lemma}{theorem}
\newtheorem{lemma}[lemma]{Lemma}
\aliascntresetthe{lemma}
\crefname{lemma}{Lemma}{Lemmas}

\newaliascnt{corollary}{theorem}

\aliascntresetthe{corollary}
\crefname{corollary}{Corollary}{Corollaries}

\theoremstyle{definition}
\newaliascnt{definition}{theorem}
\newtheorem{definition}[definition]{Definition}
\aliascntresetthe{definition}
\crefname{definition}{Definition}{Definitions}

\theoremstyle{remark}
\newaliascnt{remark}{theorem}

\aliascntresetthe{remark}
\crefname{remark}{Remark}{Remarks}

\renewcommand{\le}{\leqslant}
\renewcommand{\ge}{\geqslant}
\renewcommand{\leq}{\leqslant}
\renewcommand{\geq}{\geqslant}

\DeclarePairedDelimiter{\norm}{\lVert}{\rVert}
\DeclarePairedDelimiter{\abs}{\lvert}{\rvert}

\newcommand{\R}{\mathbb{R}}
\newcommand{\E}{\mathbb{E}}
\newcommand{\inner}[2]{\left\langle #1,#2\right\rangle}
\newcommand{\HS}{\mathrm{HS}}
\newcommand{\Id}{\mathrm{Id}}
\newcommand{\dd}{\,\mathrm{d}}

\newcommand{\mc}{\mathcal}

\DeclareMathOperator{\Var}{Var}
\DeclareMathOperator{\Cov}{Cov}
\DeclareMathOperator{\Tr}{Tr}

\title{The KLS constant is $O(\log^{1/4} n)$}
\author[Letwin]{Brayden Letwin}
\address{Department of Mathematics, University of Washington, Seattle, Washington 98195}
\email{letwin@uw.edu}

\begin{document}

\begin{abstract}
We confirm the Kannan--Lov\'asz--Simonovits conjecture for quadratic forms: if $X \sim \mu$ is an isotropic log-concave random vector in $\R^n$ then for any symmetric matrix $M$ one has that
\[
  \Var_{X \sim \mu}(\inner{MX}{X}) \le 2 \cdot \E_{X \sim \mu} \abs{\nabla \inner{MX}{X}}^2.
\]
As an application, we apply the above to $M = \E_{X \sim \mu}(\inner{X}{\theta} X \otimes X)$ for $\theta \in S^{n-1}$ and show that
the Kannan--Lov\'asz--Simonovits constant $\psi_n$ satisfies that
\[
\psi_n\le C\cdot\log^{1/4} n
\]
for some absolute constant $C > 0$.
\end{abstract}

\maketitle

\section{Introduction}\label{sec:1}
The Kannan--Lov\'asz--Simonovits (KLS) conjecture, introduced by Kannan,
Lov\'asz, and Simonovits \cite{KLS95}, is one of the
central open problems in asymptotic convex geometry.
In one of its equivalent (up to absolute constants) formulations,
it asserts that every isotropic log-concave probability measure $\mu$ on $\R^n$ satisfies a Poincar\'e inequality
\[
  \Var_{X \sim \mu}(f(X)) \le C \cdot \E_{X \sim \mu} \abs{\nabla f(X)}^2,
\]
for every locally Lipschitz function $f : \R^n \to \R$, where $C > 0$ is an absolute constant. Here $X \sim \mu$ denotes a random vector $X$
with law equal to $\mu$. For general background in the theory of asymptotic convex geometry we refer to \cite{AGM15,AGM21,BGVV14}, and for treatise
specifically centered around the KLS conjecture, we refer to \cite{AB15,LV19,KL25b}.

Besides its intrinsic interest, the conjecture is a cornerstone of high-dimensional probability, with
far-reaching applications and connections to fundamental problems in
high-dimensional geometry, information theory, statistics,
and theoretical computer science.
The conjecture grew out of the "localization method", which has grown to be modernly fruitful in computer science and its application to
random walks on convex bodies \cite{LS93,
KLS95,KLS97}; its algorithmic
role, and role in sampling log-concave distributions is developed further in
\cite{Chewi2026,CE25}. A transport approach to understanding KLS, based on
the F\"ollmer process, was also developed in
\cite{MS24}.

Recall that a probability density $\rho : \R^n \to [0, \infty)$ is log-concave if its support $\{x \in \R^n : \rho(x) > 0\}$ is a convex set and $\log \rho$ is concave on the support of $\rho$. 
Given a probability measure $\mu$ on $\R^n$, we say that $\mu$ is log-concave if, for any compact subsets $A$ and $B$ of $\R^n$ and $0 \le \lambda \le 1$ one has that
\[
  \mu(\lambda A + (1-\lambda)B) \ge \mu(A)^{\lambda}\mu(B)^{1-\lambda}.
\]
By a theorem of Borell~\cite{Borell1975}, a probability measure
$\mu$ on $\R^n$ is log-concave if and only if, upon writing
$H=\operatorname{aff}(\operatorname{supp}\mu)$, it has a log-concave density
$\rho:H\to[0,\infty)$ with respect to Lebesgue measure on $H$.
The uniform measure on a convex body, as well as a standard Gaussian measure are both canonical examples of a log-concave measure. If $X \sim \mu$ is a random vector in $\R^n$, we say that $X$ is log-concave if its law
is log-concave. Furthermore, we say that $X$ (or $\mu$ if $X \sim \mu$) is isotropic if it has zero mean and identity covariance:
\[
  \E_{X \sim \mu}(X) = 0, \quad \text{and} \quad \Cov_{X \sim \mu}(X) = \Id.
\]
The \textit{Poincar\'e constant} of a random vector $X$ with law $\mu$ in $\R^n$,
denoted by $C_P(\mu)$, is the infimum over all constants $C \geq 0$
such that, for every locally Lipschitz function
$f:\R^n\to\R$ satisfying
\[
    \E_{X \sim \mu}\abs{\nabla f(X)}^2\leq 1,
\]
one has that
\begin{equation}\label{eq:1.1}
\Var_{X \sim \mu}\bigl(f(X)\bigr)
    \leq
    C \cdot \E_{X \sim \mu}\abs{\nabla f(X)}^2.
\end{equation}
Let $\mu$ be a log-concave probability measure on $\R^n$ with density
$\rho$. Its \textit{isoperimetric constant} $\psi_\mu$ is given by
\[
  \frac{1}{\psi_\mu}
  =
  \inf_A
  \frac{\int_{\partial A}\rho(x)\dd\mathcal{H}^{n-1}(x)}
       {\min\{\mu(A),1-\mu(A)\}},
\]
where the infimum runs over all open sets $A\subseteq\R^n$ with smooth
boundary and $0<\mu(A)<1$. The \textit{Kannan--Lov\'asz--Simonovits (KLS)} constant is defined by
\[
  \psi_n
  =
  \sup_\mu\psi_\mu,
\]
where the supremum runs over all isotropic, log-concave probability measures $\mu$ on $\R^n$. First, observe that by testing \cref{eq:1.1} against linear functions $\inner{\cdot}{\theta}$ for $\theta \in S^{n-1}$ we see that $C_P(\mu) \ge 1$. Importantly, the isoperimetric constant and Poincar\'e constant
are comparable up to absolute constants. Cheeger's inequality
\cite{Cheeger1970} gives one direction, while the reverse
inequalities of Buser \cite{Buser1982} and
Ledoux \cite{Ledoux2004} give the other. De Ponti and Mondino
\cite{DM21} improved the upper bound for this comparison, so that one ultimately has
\[
  \frac{1}{4}
  \leq
  \frac{\psi_\mu^2}{C_P(\mu)}
  \leq
  \pi.
\]
Such comparisons were also shown by Milman under a broader equivalence of important parameters assuming log-concavity: Poincare inequalities, isoperimetric constants, exponential concentration, and uniform tail
bounds for Lipschitz functions control one another up to universal constants, we refer to
\cite{Milman2009} for more details.

As discussed above, the KLS conjecture asserts that every isotropic
log-concave probability measure $\mu$ on $\R^n$ satisfies
$C_P(\mu)\leq C$, where $C>0$ is an absolute constant independent of
$n$. Equivalently, $\sup_{n \ge 2}\psi_n\leq C$.
The KLS conjecture is known for $\ell_p^n$-balls and for broad subclasses of generalized Orlicz balls; see Kolesnikov and Milman \cite{KM14,KM18}.

The conjecture is closely connected with thin-shell
concentration
\cite{Bobkov2007,Paouris2006,Fleury2010,GM11,
KL25c}, Gaussian behavior of linear marginals of convex bodies
\cite{ABP03,BK03,Klartag2007CLT,
Klartag2007PowerLaw,EK08,LM26}, and was known to imply the now solved \textit{slicing problem} of Bourgain:
\cite{Bourgain1986,Ball1988,EK11,KL25a,Bizeul2025}.
Links between KLS and a generalized central-limit conjecture were developed up to polylogarithmic factors in $n$
in \cite{JLV20}. Further quantitative central limit theorems
for sums of high-dimensional log-concave random vectors appear in
\cite{EMZ20,FK24}.

With a series of breakthroughs initiated by Eldan's method of stochastic-localization \cite{Eldan2013}, Lee and Vempala
\cite{LV24} proved the bound
$\psi_n \le C \cdot n^{1/4}$. Chen
\cite{Chen2021} subsequently proved
that
$\psi_n \le \exp(C\sqrt{\log n\cdot\log\log n})$.
Klartag and Lehec \cite{KL22} then gave the first
polylogarithmic bound $\psi_n \le C\cdot\log^5 n$, which was improved to
$\psi_n \le C\cdot\log^{3.2226}n$ by Jambulapati, Lee, and Vempala
\cite{JLV22}. Afterwards, Klartag
\cite{Klartag2023Logarithmic} proved the bound
$\psi_n \le C\cdot\sqrt{\log n}$, which was the previously best known bound. The
purpose of this paper is to give a further improvement beyond this bound.

\begin{theorem} \label{thm:1.1}
Let $n \ge 2$. Then
\[
\psi_n \le C \cdot \log^{1/4} n.
\]
\end{theorem}
The bound in \cref{thm:1.1} follows immediately from the following Poincar\'e inequality for quadratic forms \cref{thm:1.2}. This is clear to experts, but we give a brief proof before focusing on \cref{thm:1.2} for the remainder of the paper.

\begin{theorem} \label{thm:1.2}
Suppose that $X \sim \mu$ is an isotropic log-concave random vector in $\R^n$. Then for any symmetric matrix $M$ one has that
\begin{equation} \label{eq:1.2}
\Var_{X \sim \mu}(\inner{MX}{X}) \le 2 \cdot \E_{X \sim \mu} \abs{\nabla \inner{MX}{X}}^2.
\end{equation}
\end{theorem}
In the terminology of Bobkov, Chistyakov, and G\"otze
\cite{BCG20}, \cref{thm:1.2} gives their second-order
correlation condition with constant $8$. Consequently, for centrally
symmetric isotropic log-concave laws their theorem gives an $O(\log n/n)$ bound for
the average Kolmogorov distance between one-dimensional marginals and the
standard Gaussian law.
The constant $2$ appearing in \cref{eq:1.2} is sharp. A simple exercise shows that it is achieved by an exponential random vector in isotropic position and with $M = \Id$. 
That is, let $X = (X_1, \ldots, X_n)$ where each $X_i$ are independent coordinates with $X_i = Y_i - 1$ and $Y_i$ is an exponential random variable with parameter one. Then the constant $2$ above is attained exactly.
We will now prove \cref{thm:1.1} provided that \cref{thm:1.2} holds.
\begin{proof}[Proof of \cref{thm:1.1}]
  Following Eldan~\cite[Equation~6]{Eldan2013}, and using the notation of
  Klartag and Lehec~\cite[Equation~23]{KL22}, define
  the parameter $\kappa_n$ by
  \[
    \kappa_n = \sup_{\mu} \sup_{\theta \in S^{n-1}} \norm{\E_{X \sim \mu}(\inner{X}{\theta} X \otimes X)}_{\HS},
  \]
  where the supremum is taken over all isotropic log-concave random vectors in $\R^n$ with law $\mu$. What we will do is apply \cref{thm:1.2} to $M = \E_{X \sim \mu}(\inner{X}{\theta} X \otimes X)$ for some fixed isotropic log-concave random vector $X$ in $\R^n$ and $\theta \in S^{n-1}$. This matrix $M$ is clearly symmetric, and its definition gives
  \[
    \begin{aligned}
      \norm{M}_{\HS}^2
      &= \Tr(M^2)
       = \E_{X\sim\mu}\!\left[
           \inner{X}{\theta}\inner{MX}{X}
         \right] \\
      &= \E_{X\sim\mu}\!\left[
           \inner{X}{\theta}
           \bigl(\inner{MX}{X}-\Tr M\bigr)
         \right],
    \end{aligned}
  \]
  where we used $\E_{X\sim\mu}\inner{X}{\theta}=0$ and
  $\E_{X\sim\mu}\inner{MX}{X}=\Tr M$. Therefore, Cauchy--Schwarz gives
  \[
    \norm{M}_{\HS}^4
    \leq
    \Var_{X\sim\mu}\inner{X}{\theta}
    \cdot
    \Var_{X\sim\mu}\bigl(\inner{MX}{X}\bigr).
  \]
  Since $\Var_{X \sim \mu} \inner{X}{\theta} = 1$ as $\Cov_{X \sim \mu}(X) = \Id$, we may now apply \cref{thm:1.2} to conclude using $\nabla\inner{MX}{X} = 2MX$ that
  \[
    \norm{M}_{\HS}^4 \le 8 \norm{M}_{\HS}^2,
  \]
  and consequently we see that $\kappa_n \le 2\sqrt{2}$. Klartag's
  improved Lichnerowicz inequality
  \cite[Theorem~1.3, Corollary~3.2, and the discussion following
  Equation~3.13]{Klartag2023Logarithmic} implies for
  any isotropic log-concave measure $\mu$ we have
  \[
  C_P(\mu) \le C \cdot \kappa_n \sqrt{\log n}.
  \]
  Formally speaking, this bound is not the final bound that Klartag uses, but tracing
  his inequalities shows that such a bound appears from his result. Then using $\kappa_n \le 2 \sqrt{2}$ along with
  $\psi_n^2 \le C \cdot \sup_{\mu} C_P(\mu)$ lets us conclude
  \cref{thm:1.1}.
\end{proof}

\subsection{Notation}

Before beginning the proof of \cref{thm:1.2}, we first state the
notation that will be used throughout. $C, c, C_1, C_2, \ldots > 0,$ denote 
absolute constants which may change between lines. For $x,y\in\R^n$, we write $\inner{x}{y}$ for the standard
Euclidean inner product and $\abs{x}= \sqrt{\inner{x}{x}}$ for the
corresponding Euclidean norm. We write $(e_1,\ldots,e_n)$ for the
standard basis of $\R^n$ and $S^{n-1}$ for the Euclidean unit sphere.
The tensor product $x\otimes y$ is the rank-one matrix determined by
\[
  (x\otimes y)_{ij}=x_i y_j.
\]
Expectations of random vectors and matrices are always taken entrywise.

We use standard probabilistic notation. The expression $X\sim\mu$
means that the random vector $X$ has law $\mu$. As stated above, $X\sim\mu$ is isotropic precisely when
$\E_{X\sim\mu}X=0$ and $\Cov_{X\sim\mu}(X)=\Id$.

We write $A^T$ for the transpose of a matrix $A$,
$\Tr(A)$ for its trace, and $\Id$ for the identity matrix. The notation
$A\succeq0$ means that $A$ is symmetric and
$\inner{Ax}{x}\geq0$ for every $x\in\R^n$. When $A\succeq0$, the
notation $A^{1/2}$ always refers to its unique positive semidefinite
square root. The
Hilbert--Schmidt inner product and its associated norm are
\[
  \inner{A}{B}_{\HS}=\Tr(A^TB),
  \qquad
  \norm{A}_{\HS}^2=\Tr(A^TA).
\]
In particular, if $A$ is symmetric, then
$\norm{A}_{\HS}^2=\Tr(A^2)$. Since the matrices appearing below do not
generally commute, the order of the factors in every matrix product is
the order in which they are written. We will use cyclicity of the trace,
\[
  \Tr(A_1A_2\cdots A_k)=\Tr(A_2\cdots A_kA_1).
\]

For a smooth function $f:\R^n\to\R$, we write $\partial_i f$ for its
$i$th partial derivative, $\nabla f$ for its gradient, and
$\nabla^2f=(\partial_{ij}f)_{i,j=1}^n$ for its Hessian. We also write
$\partial_{ijk}f$ for a third derivative. When $\nabla^2 f$ is invertible, we write
$(f^{ij})=(\nabla^2 f)^{-1}$. It will be quite convenient to adopt Einstein summation notation
throughout the proof, so that in particular, we have
\[
  f^{ia}\partial_{aj}f=\delta_i^j.
\]
Sometimes though, for clarity we will simply write in regular summation notation. The notation will be
clear at each stage.
We will also repeatedly transport measures through maps. If
$T:\R^n\to\R^n$ is measurable, then $T_{\#}\mu$ denotes the pushforward
of $\mu$ under $T$, defined by
\[
  (T_{\#}\mu)(E)=\mu(T^{-1}(E))
\]
for every Borel set $E\subseteq\R^n$. Equivalently, if $X\sim\mu$, then
$T(X)\sim T_{\#}\mu$.

For a symmetric matrix $M$, we write
\[
  \abs{M}=(M^2)^{1/2}.
\]
This is the positive semidefinite matrix obtained by replacing every
eigenvalue of $M$ by its absolute value. When $M$ is invertible, we
further write
\[
  \operatorname{sgn}(M)=M\abs{M}^{-1}.
\]
It follows directly from the spectral theorem that
\[
  M=\operatorname{sgn}(M)\abs{M}
  =\abs{M}\operatorname{sgn}(M),
  \qquad
  \operatorname{sgn}(M)^T=\operatorname{sgn}(M),
  \qquad
  \operatorname{sgn}(M)^2=\Id.
\]
The meaning of $\abs{\cdot}$ will always be clear from its argument: it denotes the
usual absolute value for scalars, the Euclidean norm for vectors, and
the matrix absolute value above for symmetric matrices. 
\subsection{AI Disclosure}
The author for some time had been experimenting with moment measures and their implications on expected values of
particular functionals of isotropic log-concave random vectors in $\R^n$. In doing this, the author noticed that
the expected third moment tensor of an isotropic log-concave random vector in $\R^n$ could be written in terms of the expected
third derivative of the associated moment map. ChatGPT $5.6$ Pro then drew the authors attention to a variant of \cref{thm:2.5} with $B = \Id$ after experimenting with the the Monge--Amp\`ere equation, from which the author noticed that the proof would hold for general symmetric matrices $B$. ChatGPT $5.6$ Pro also wrote all of the regularity arguments
appearing in the appendix.

\section{\texorpdfstring{Proof of \cref{thm:1.2}}{Proof of Theorem 1.2}}\label{sec:2}

\subsection{The moment-measure argument}

To keep the main argument focused, we work throughout this section under
the regularity assumptions permitted by
\cref{lem:a1}.

Our proof uses moment measures. Given a convex function
$\varphi:\R^n\to\R\cup\{+\infty\}$ such that
\[
  0<\int_{\R^n}e^{-\varphi(y)}\,\dd y<\infty,
\]
its moment measure is the Borel measure $\mu$ on $\R^n$ defined by
\[
  \mu
  =
  (\nabla\varphi)_{\#}
  \left(e^{-\varphi(y)}\,\dd y\right).
\]
Equivalently, for every nonnegative Borel function $f$, or every Borel
function $f\in L^1(\mu)$,
\[
  \int_{\R^n}f(x)\,\dd\mu(x)
  =
  \int_{\R^n}f(\nabla\varphi(y))e^{-\varphi(y)}\,\dd y.
\]
This framework is due to Cordero-Erausquin
and Klartag; see
\cite{CK15}. In particular, we will specifically work in the
setting developed by Klartag
\cite{Klartag2014}. From these papers, our main tools will be \cref{lem:2.1} and \cref{lem:2.2}.
For related work on moment measures and their applications, see
\cite{Klartag2013Poincare,Kolesnikov2014,KM16,KK17a,
Santambrogio2016,FM22}.

The moment-measures theorem of Cordero--Erausquin and Klartag, which is the main construction that this paper follows, states the following:
\begin{lemma}[Cordero--Erausquin and Klartag
{\cite[Theorem~2]{CK15}}]
\label{lem:2.1}
Let $\mu$ be a Borel measure on $\R^n$ such that
$0<\mu(\R^n)<\infty$, its support is not contained in a proper linear
subspace, and its barycenter is at the origin. Then there exists an
essentially continuous convex function
\[
    \varphi:\R^n\longrightarrow\R\cup\{+\infty\}
\]
such that
\[
    (\nabla\varphi)_{\#}
    \left(e^{-\varphi(y)}\,\dd y\right)=\mu.
\]
Moreover, if $\varphi_0$ and $\varphi_1$ both satisfy the conclusion,
then there is a $y_0\in\R^n$ such that, for all $y\in\R^n$,
\[
    \varphi_1(y)=\varphi_0(y-y_0).
\]
\end{lemma}
We call $\varphi$ the moment map of $\mu$. The properties of the moment map used below
are summarized in \cref{lem:2.2}. Let $\mu$ be an isotropic log-concave
measure on $\R^n$. Concretely, we may assume that $\mu$ is of the form
\[
    \dd\mu(x)
    =\rho(x)\,\dd x
    =e^{-V(x)}\mathbf{1}_{K}(x)\,\dd x,
\]
where $K\subset\R^n$ is bounded, open, and convex, and $V$ is smooth and
convex on an open neighborhood of $\overline K$.

\begin{lemma}[Klartag
{\cite{Klartag2014}}]
\label{lem:2.2}
Under the regularity assumptions of \cref{lem:a1},
the convex function $\varphi$
given by \cref{lem:2.1} is finite and smooth on $\R^n$,
and its Hessian $\nabla^2\varphi$ is positive definite at every point.
Moreover, the map
\[
    \nabla\varphi:\R^n\longrightarrow K
\]
is a diffeomorphism. For all $y\in\R^n$, the Monge--Amp\`ere equation
\[
    \rho(\nabla\varphi(y))\det\nabla^2\varphi(y)
    =e^{-\varphi(y)}
\]
holds pointwise. Equivalently, since $\nabla^2\varphi$ is positive
definite, for all $y\in\R^n$,
\[
    \log\det\nabla^2\varphi(y)
    =-\varphi(y)+V(\nabla\varphi(y)).
\]
\end{lemma}
Before going through the remainder of the proof, we will explain our strategy at hand, and the motivation behind the construction. Set
\[
    \dd\nu(y)=e^{-\varphi(y)}\,\dd y,
\]
so that $(\nabla\varphi)_{\#}\nu=\mu$. Since $\mu$ is a probability
measure, we see that $\nu$ is a probability
measure. In particular, if $Y$ has law $\nu$, then
$\nabla\varphi(Y)$ has law $\mu$. We will pass to understand the problem under $\nu$. Since $\nabla\varphi(\R^n)=K$ and $K$ is bounded,
$\nabla\varphi$ is bounded. Moreover,
since $\nabla^2\varphi(y)\succeq0$, \cite[Theorem~1]{Klartag2014}
gives, for every $y\in\R^n$,
\[
  0\preceq\nabla^2\varphi(y)\preceq2\sup_{x\in K}\abs{x}^2\Id.
\]
Thus both $\nabla\varphi$ and $\nabla^2\varphi$ are bounded on $\R^n$.

By \cref{lem:a3} we may use a standard bump function argument to show that if $f,g:\R^n\to\R$ are smooth and
$(\partial_i f)g$, $f\partial_i g$, $fg\partial_i\varphi$, and $fg$
belong to $L^1(\nu)$, then
\begin{equation}\label{eq:2.1}
  \E_{Y\sim\nu}\!\left[(\partial_i f)(Y)g(Y)\right]
  =
  \E_{Y\sim\nu}\!\left[
    f(Y)\bigl(\partial_i\varphi(Y)g(Y)-\partial_i g(Y)\bigr)
  \right].
\end{equation}
When $g\equiv1$ and the remaining terms are integrable, this becomes
\[
  \E_{Y\sim\nu}\partial_i f(Y)
  =
  \E_{Y\sim\nu}\!\left[f(Y)\partial_i\varphi(Y)\right].
\]
We also use the weighted Laplace-type operator
\[
  \mc{L}f
  =
  \varphi^{ij}\partial_{ij}f
  -(V_i\circ\nabla\varphi)\partial_i f.
\]
This operator is a Markov diffusion operator, which is symmetric on $L^2(\nu)$. The coordinate formula is computed in
\cite[Lemma~5]{Klartag2014}; see also
\cite{BGL14} for general background on Markov diffusion operators. The key defining formula about $\mc{L}$ is that one has
\begin{equation}\label{eq:2.2}
  \E_{Y\sim\nu}\!\left[(\mc{L}f)(Y)g(Y)\right]
  =
  -\E_{Y\sim\nu}\!\left[
    \varphi^{ij}(Y)(\partial_i f)(Y)(\partial_j g)(Y)
  \right]
\end{equation}
whenever $f$ or $g$ is compactly supported, this is
\cite[Equations~53--55]{Klartag2014}. In particular, for
$f\in C_c^\infty(\R^n)$, taking $g\equiv1$ gives
\[
  \E_{Y\sim\nu}\mc{L}f(Y)=0.
\]
The case where both $f$ and $g$ are not known to be compact will be used one time in this paper, and is justified by
\cref{lem:a5}; see the argument following \cref{eq:2.10} for the case where one must invoke \cref{lem:a5}.

Moving back to understanding $\nu$, we may lose in the fact that $\nu$ is isotropic. Wonderfully though,
we have the following as a suitable replacement which states that $\nabla^2 \varphi$ is on average the identity under $\nu$, which in
some sense can be seen as a suitable replacement.
\begin{lemma}\label{lem:2.3}
  One has that, with $\nu$ defined as above
  \[
    \E_{Y \sim \nu} \nabla^2 \varphi(Y) = \Id.
  \]
\end{lemma}
\begin{proof}
Fix $1\leq i,j\leq n$. The functions $\partial_i\varphi$,
$\partial_{ij}\varphi$, and
$\partial_i\varphi\,\partial_j\varphi$ are bounded, and $\nu$ is a
probability measure. Hence the hypotheses of \cref{lem:a3}
are satisfied. Applying \cref{eq:2.1} in coordinate
$j$, with $f=\partial_i\varphi$ and $g\equiv1$, gives
\begin{align*}
  \E_{Y\sim\nu}\partial_{ij}\varphi(Y)
  &=\E_{Y\sim\nu}
    \bigl[\partial_i\varphi(Y)\partial_j\varphi(Y)\bigr].
\end{align*}
Since
$(\nabla\varphi)_{\#}\nu=\mu$, the random vector
$\nabla\varphi(Y)$ has law $\mu$. Isotropy of $\mu$ therefore yields
\[
  \E_{Y\sim\nu}
  \bigl[\partial_i\varphi(Y)\partial_j\varphi(Y)\bigr]
  =\E_{X\sim\mu}[X_iX_j]
  =\delta_{ij}.
\]
The result follows entrywise.
\end{proof}

We now turn to the Monge--Amp\`ere equation from \cref{lem:2.2}. \Cref{thm:2.5} is our
key result that allows us to obtain bounds on the Poincar\'e constant
for quadratic forms, and hence on $\psi_n$. In order to do this, we will differentiate the Monge--Amp\`ere equation
and obtain \cref{lem:2.4}, which will be crucial for the proof of \cref{thm:2.5}.

We also note that the following
product rule holds for $\mc{L}$. For smooth
functions $f,g:\R^n\to\R$, expanding the derivatives in coordinates
gives
\begin{equation}\label{eq:2.3}
\begin{aligned}
  \mc{L}(fg)
  &=\varphi^{ij}\partial_{ij}(fg)
    -(V_i\circ\nabla\varphi)\partial_i(fg)\\
  &=f\mc{L}g+g\mc{L}f
    +2\varphi^{ij}(\partial_i f)(\partial_j g).
\end{aligned}
\end{equation}
The two mixed second-derivative terms agree because
$(\varphi^{ij})$ is symmetric.
The next identity follows by differentiating the Monge--Amp\`ere
equation from \cref{lem:2.2} twice.
\begin{lemma}\label{lem:2.4}
For every $1\leq i,j\leq n$, one has
\[
  \mc{L}(\partial_{ij}\varphi)+\partial_{ij}\varphi
  =\partial_{ia}\varphi\,V_{ab}\partial_{bj}\varphi
  +\varphi^{ac}\varphi^{bd}\partial_{abi}\varphi\,
    \partial_{cdj}\varphi.
\]
Here $V_{ab}$ is evaluated at $\nabla\varphi$. Moreover, pointwise, the
two matrices
\[
  \left(
    \partial_{ia}\varphi\,V_{ab}\partial_{bj}\varphi
  \right)_{i,j=1}^n
\]
and
\[
  \bigl(
    \varphi^{ac}\varphi^{bd}\partial_{abi}\varphi\,
    \partial_{cdj}\varphi
  \bigr)_{i,j=1}^n
\]
are positive semidefinite.
\end{lemma}
\begin{proof}
For ease of notation: all derivatives of $V$
below are evaluated at $\nabla \varphi$.
Since all the terms in the Monge--Amp\`ere equation appearing in \cref{lem:2.2} are positive, we may take logarithms and differentiate
\[
  \log\det\nabla^2\varphi=-\varphi+V(\nabla\varphi)
\]
twice. We first differentiate with respect to the $i$th coordinate. By
Jacobi's formula and the chain rule,
\[
  \partial_i\log\det\nabla^2\varphi
  =\Tr\!\left(
    (\nabla^2\varphi)^{-1}
    \partial_i(\nabla^2\varphi)
  \right)
  =\varphi^{ab}\partial_{abi}\varphi,
  \qquad
  \partial_i\bigl[V(\nabla\varphi)\bigr]
  =V_a\partial_{ai}\varphi.
\]
We therefore obtain the first-derivative of the Monge--Amp\`ere equation
\begin{equation}\label{eq:2.4}
  \varphi^{ab}\partial_{abi}\varphi
  =-\partial_i\varphi+V_a\partial_{ai}\varphi.
\end{equation}
Recall that $(\varphi^{ij}) = (\nabla^2 \varphi)^{-1}$, so differentiating the equation
$\varphi^{ia}\partial_{aj}\varphi=\delta^i_j$ gives, for every $k$,
\begin{equation}\label{eq:2.5}
  \partial_k\varphi^{ij}
  =-\varphi^{ia}\partial_{abk}\varphi\,\varphi^{bj}.
\end{equation}
We next return to the first-derivative of the Monge--Amp\`ere equation \cref{eq:2.4} and differentiate it
with respect to the $j$th coordinate. Using \cref{eq:2.5}, the
derivative of the left-hand side is
\begin{align*}
  \partial_j\left(
    \varphi^{ab}\partial_{abi}\varphi\right)
  &=\varphi^{ab}\partial_{abij}\varphi
    +(\partial_j\varphi^{ab})\partial_{abi}\varphi\\
  &=\varphi^{ab}\partial_{abij}\varphi
    -\varphi^{ac}\varphi^{bd}
      \partial_{abi}\varphi\,\partial_{cdj}\varphi.
\end{align*}
On the other hand, the derivative of the right-hand side of
\cref{eq:2.4} is
\begin{align*}
  \partial_j\left(
    -\partial_i\varphi+V_a\partial_{ai}\varphi\right)
  &=-\partial_{ij}\varphi
    +V_{ab}\partial_{bj}\varphi\,\partial_{ai}\varphi
    +V_a\partial_{aij}\varphi.
\end{align*}
Equating these two expressions gives
\begin{align*}
  \varphi^{ab}\partial_{abij}\varphi
  -\varphi^{ac}\varphi^{bd}
    \partial_{abi}\varphi\,\partial_{cdj}\varphi
  =-\partial_{ij}\varphi
    +V_{ab}\partial_{ai}\varphi\,\partial_{bj}\varphi
    +V_a\partial_{aij}\varphi.
\end{align*}
Note now that the two terms that are linear in derivatives of
$\partial_{ij}\varphi$ combine exactly as $\mc{L}$
\[
  \varphi^{ab}\partial_{abij}\varphi
  -V_a\partial_{aij}\varphi
  =\mc{L}(\partial_{ij}\varphi).
\]
Thus, we obtain the following equation as the second-derivative of the Monge--Amp\`ere equation
\begin{equation}\label{eq:2.6}
  \mc{L}(\partial_{ij}\varphi)+\partial_{ij}\varphi
  =\partial_{ia}\varphi\,V_{ab}\partial_{bj}\varphi
  +\varphi^{ac}\varphi^{bd}\partial_{abi}\varphi\,
    \partial_{cdj}\varphi.
\end{equation}
The convexity of $V$ implies that the matrix
\[
  \left(\partial_{ia}\varphi\,V_{ab}
  \partial_{bj}\varphi\right)_{i,j=1}^n
  =\nabla^2\varphi\cdot\nabla^2V(\nabla\varphi)\cdot\nabla^2\varphi
\]
is positive semidefinite. The matrix
\[
  \left(
    \varphi^{ac}\varphi^{bd}\partial_{abi}\varphi\,
    \partial_{cdj}\varphi
  \right)_{i,j=1}^n
\]
is positive semidefinite as well, since it is the Gram matrix, with
respect to the Hilbert--Schmidt inner product for $1 \le i \le n$ of the matrices
\[
  (\nabla^2\varphi)^{-1/2}\partial_i\nabla^2\varphi
  (\nabla^2\varphi)^{-1/2}.
\]
Indeed,
\begin{align*}
&\left\langle
(\nabla^2\varphi)^{-1/2}
\partial_i\nabla^2\varphi
(\nabla^2\varphi)^{-1/2},
\,
(\nabla^2\varphi)^{-1/2}
\partial_j\nabla^2\varphi
(\nabla^2\varphi)^{-1/2}
\right\rangle_{\mathrm{HS}}
\\[4pt]
&=
\operatorname{Tr}\!\left(
(\nabla^2\varphi)^{-1}
\partial_i\nabla^2\varphi
(\nabla^2\varphi)^{-1}
\partial_j\nabla^2\varphi
\right)
\\[4pt]
&=
\varphi^{ac}\varphi^{bd}
\partial_{abi}\varphi\,
\partial_{cdj}\varphi.\qedhere
\end{align*}
\end{proof}

\begin{theorem}\label{thm:2.5}\footnote{The motivation to obtain \cref{thm:2.5} came from reverse engineering the remainder of the proof to come, which used a promising trick \cref{eq:2.21} with a $H^{-1}$ inequality combined with the knowledge that $\tau_\mu
    =\nabla^2\varphi\circ(\nabla\varphi)^{-1}$ is a Stein kernel for $\mu$ (this observation is due to Fathi \cite{Fathi2019}). The author considered studying the following candidates: $(\nabla^2\varphi)^{1/2}B(\nabla^2\varphi)^{1/2}$ has the desired squared norm $\Tr(B\nabla^2\varphi B\nabla^2\varphi)$, but no explicit mean or simple derivative. However, for $B\succeq0$, the matrix $B^{1/2}\nabla^2\varphi B^{1/2}$ has mean $B$, differentiates well, and has the desired squared norm. ChatGPT 5.6 Pro then differentiated the Monge--Amp\`ere equation and combined all the tools together.}
  With $\nu$ defined as above, for any symmetric matrix $B$, one has
  \[
    \E_{Y\sim\nu}
    \Tr\!\left(B\nabla^2\varphi(Y)B\nabla^2\varphi(Y)\right)
    \le
    2\Tr(B^2).
  \]
\end{theorem}
\begin{proof}
We will start by considering positive semidefinite symmetric matrices $B$. 
The indefinite case will follow promptly after. Let $B\succeq0$. We apply the product rule in \cref{eq:2.3} to the quantity in the
statement of the theorem. In coordinates,
\[
  \Tr\!\left(
    B\nabla^2\varphi B\nabla^2\varphi
  \right)
  =
  B_{ij}\partial_{jk}\varphi\,
  B_{k\ell}\partial_{\ell i}\varphi.
\]
Since the entries of $B$ are constant, applying $\mc{L}$ gives
\[
  \mc{L}\!\left[\Tr\!\left(B\nabla^2\varphi B\nabla^2\varphi\right)\right]
  =B_{ij}B_{k\ell}\,\mc{L}\!\left(
    \partial_{jk}\varphi\,\partial_{\ell i}\varphi
  \right).
\]
The product rule in \cref{eq:2.3} now yields
\begin{equation}\label{eq:2.7}
  \mc{L}\!\left[\Tr\!\left(B\nabla^2\varphi B\nabla^2\varphi\right)\right]
  =B_{ij}B_{k\ell}\mc{L}(\partial_{jk}\varphi)\partial_{\ell i}\varphi
  +B_{ij}B_{k\ell}\partial_{jk}\varphi\mc{L}(\partial_{\ell i}\varphi)
  +2\varphi^{ab}B_{ij}B_{k\ell}\partial_{ajk}\varphi\partial_{b\ell i}\varphi.
\end{equation}
The first term on the right-hand side of \cref{eq:2.7} is
\[
  \Tr\!\left(
    B\mc{L}(\nabla^2\varphi)
    B\nabla^2\varphi
  \right),
\]
where $\mc{L}$ acts entrywise on $\nabla^2\varphi$. The second term on
the right-hand side of \cref{eq:2.7} is
\[
  \Tr\!\left(
    B\nabla^2\varphi
    B\mc{L}(\nabla^2\varphi)
  \right).
\]
These two traces are equal. Indeed, transposing the matrix inside the
first trace and then using cyclicity of the trace gives
\[
  \Tr\!\left(
    B\mc{L}(\nabla^2\varphi)
    B\nabla^2\varphi
  \right)
  =\Tr\!\left(
    \nabla^2\varphi B
    \mc{L}(\nabla^2\varphi)B
  \right)
  =\Tr\!\left(
    \bigl(B\nabla^2\varphi B\bigr)\cdot
    \mc{L}(\nabla^2\varphi)
  \right).
\]
The third term on the right-hand side of \cref{eq:2.7} may be written as
\[
  2\varphi^{ab}
  \Tr\!\left(
    \bigl(B\partial_a(\nabla^2\varphi)B\bigr)\cdot
    \partial_b(\nabla^2\varphi)
  \right).
\]
We have therefore shown that
\begin{equation}\label{eq:2.8}
  \mc{L}\!\left[\Tr\!\left(B\nabla^2\varphi B\nabla^2\varphi\right)\right]
  =2\Tr\!\left(
    \bigl(B\nabla^2\varphi B\bigr)\cdot
    \mc{L}(\nabla^2\varphi)
  \right)
  +2\varphi^{ab}
  \Tr\!\left(
    \bigl(B\partial_a(\nabla^2\varphi)B\bigr)\cdot
    \partial_b(\nabla^2\varphi)
  \right).
\end{equation}

We next use \cref{lem:2.4}, specifically \cref{eq:2.6}. Since that
identity holds for every pair of indices, it says entrywise that
\begin{equation}\label{eq:2.9}
  \mc{L}(\nabla^2\varphi)
  =-\nabla^2\varphi
  +\nabla^2\varphi\,
    \nabla^2V(\nabla\varphi)\,
    \nabla^2\varphi
  +
  \left(
    \varphi^{ac}\varphi^{bd}
    \partial_{abi}\varphi\,
    \partial_{cdj}\varphi
  \right)_{i,j=1}^n.
\end{equation}
Substituting \cref{eq:2.9} into \cref{eq:2.8} and using linearity of the
trace yields the pointwise identity
\begin{equation}\label{eq:2.10}
\begin{aligned}
  &\mc{L}\!\left[
    \Tr\!\left(
      B\nabla^2\varphi B\nabla^2\varphi
    \right)
  \right]\\
  &=-2\Tr\!\left(
    B\nabla^2\varphi B\nabla^2\varphi
  \right)\\
  &\quad+2\varphi^{ab}
  \Tr\!\left(
    \bigl(B\partial_a(\nabla^2\varphi)B\bigr)\cdot
    \partial_b(\nabla^2\varphi)
  \right)\\
  &\quad+2\Tr\!\left(
    \bigl(B\nabla^2\varphi B\bigr)\cdot
    \bigl(\nabla^2\varphi\,\nabla^2V(\nabla\varphi)\,\nabla^2\varphi\bigr)
  \right)\\
  &\quad+2\left(B\nabla^2\varphi B\right)_{ij}
  \varphi^{ac}\varphi^{bd}
  \partial_{abi}\varphi\,
  \partial_{cdj}\varphi.
\end{aligned}
\end{equation}

We first justify integrating this pointwise identity. The goal is to apply $\E_{Y \sim \nu}$ to kill the left hand side of \cref{eq:2.10} after using \cref{eq:2.2}. Recall that $\nabla^2\varphi$ is bounded. Since $B\succeq0$, we have
\[
  \Tr\!\left(
    B\nabla^2\varphi B\nabla^2\varphi
  \right)
  =\left\|
    B^{1/2}\nabla^2\varphi B^{1/2}
  \right\|_{\mathrm{HS}}^2.
\]
It follows that this trace is a bounded nonnegative function.

We next check that each of the three terms accompanying the negative
term on the right-hand side of \cref{eq:2.10} is nonnegative. For the
first such term in \cref{eq:2.10}, direct
expansion of a squared Hilbert--Schmidt norm gives
\[
  \varphi^{ab}
  \Tr\!\left(
    \bigl(B\partial_a(\nabla^2\varphi)B\bigr)\cdot
    \partial_b(\nabla^2\varphi)
  \right)
  =\sum_{r=1}^n
  \left\|
    \sum_{a=1}^n
    \bigl((\nabla^2\varphi)^{-1/2}\bigr)_{ra}
    B^{1/2}\partial_a(\nabla^2\varphi)B^{1/2}
  \right\|_{\mathrm{HS}}^2
  \geq0.
\]
For the second such term in \cref{eq:2.10}, $B\cdot\nabla^2\varphi\cdot B\succeq0$, while $\nabla^2\varphi\cdot\nabla^2V(\nabla\varphi)\cdot\nabla^2\varphi\succeq0$ by \cref{lem:2.4}. The trace of their product is therefore non-negative.
For the third such term in \cref{eq:2.10}, the matrix
\[
  \left(
    \varphi^{ac}\varphi^{bd}
    \partial_{abi}\varphi\,
    \partial_{cdj}\varphi
  \right)_{i,j=1}^n
\]
is positive semidefinite by \cref{lem:2.4}. The matrix
$B\nabla^2\varphi B$ is also positive semidefinite. The trace of their product is
\[
  \left(B\nabla^2\varphi B\right)_{ij}
  \varphi^{ac}\varphi^{bd}
  \partial_{abi}\varphi\,
  \partial_{cdj}\varphi,
\]
so the third such term in \cref{eq:2.10} is nonnegative as well. The boundedness established above gives
\[
  \Tr\!\left(
    B\nabla^2\varphi B\nabla^2\varphi
  \right)
  \in L^\infty(\nu).
\]
Moreover, the nonnegativity of the three terms just verified and
\cref{eq:2.10} give the pointwise inequality
\[
  \mc{L}\!\left[
    \Tr\!\left(
      B\nabla^2\varphi B\nabla^2\varphi
    \right)
  \right]
  +2\Tr\!\left(
    B\nabla^2\varphi B\nabla^2\varphi
  \right)
  \geq0.
\]
Consequently, \cref{lem:a5}, applied directly with $c=2$, yields
\[
  \mc{L}\!\left[
    \Tr\!\left(
      B\nabla^2\varphi B\nabla^2\varphi
    \right)
  \right]
  \in L^1(\nu),
  \qquad
  \E_{Y\sim\nu}
  \mc{L}\!\left[
    \Tr\!\left(
      B\nabla^2\varphi(Y)B\nabla^2\varphi(Y)
    \right)
  \right]
  =0.
\]
Rearranging \cref{eq:2.10} now expresses the sum of the three nonnegative terms as something in $L^1(\nu)$.
The sum is
therefore integrable, and since the terms are nonnegative, each is
integrable separately. We may thus take expectations to obtain
\[
\begin{aligned}
  0
  &=-2\cdot\E_{Y\sim\nu}
  \Tr\!\left(
    B\nabla^2\varphi(Y)B\nabla^2\varphi(Y)
  \right)\\
  &\quad+2\cdot\E_{Y\sim\nu}\!\left[
    \varphi^{ab}(Y)
    \Tr\!\left(
      \bigl(B\partial_a(\nabla^2\varphi)(Y)B\bigr)\cdot
      \partial_b(\nabla^2\varphi)(Y)
    \right)
  \right]\\
  &\quad+2\cdot\E_{Y\sim\nu}
  \Tr\!\left(
    \bigl(B\nabla^2\varphi(Y)B\bigr)\cdot
    \bigl(\nabla^2\varphi(Y)\nabla^2V(\nabla\varphi(Y))\nabla^2\varphi(Y)\bigr)
  \right)\\
  &\quad+2\cdot\E_{Y\sim\nu}\!\left[
    \left(B\nabla^2\varphi(Y)B\right)_{ij}
    \varphi^{ac}(Y)\varphi^{bd}(Y)
    \partial_{abi}\varphi(Y)\,
    \partial_{cdj}\varphi(Y)
  \right].
\end{aligned}
\]
We divide this equality by two and move its negative term to the other
side. We obtain
\begin{equation}\label{eq:2.11}
\begin{aligned}
  &\E_{Y\sim\nu}
  \Tr\!\left(
    B\nabla^2\varphi(Y)B\nabla^2\varphi(Y)
  \right)\\
  &=\E_{Y\sim\nu}\!\left[
    \varphi^{ab}(Y)
    \Tr\!\left(
      \bigl(B\partial_a(\nabla^2\varphi)(Y)B\bigr)\cdot
      \partial_b(\nabla^2\varphi)(Y)
    \right)
  \right]\\
  &\quad+\E_{Y\sim\nu}
  \Tr\!\left(
    \bigl(B\nabla^2\varphi(Y)B\bigr)\cdot
    \bigl(\nabla^2\varphi(Y)\nabla^2V(\nabla\varphi(Y))\nabla^2\varphi(Y)\bigr)
  \right)\\
  &\quad+\E_{Y\sim\nu}\!\left[
    \left(B\nabla^2\varphi(Y)B\right)_{ij}
    \varphi^{ac}(Y)\varphi^{bd}(Y)
    \partial_{abi}\varphi(Y)\,
    \partial_{cdj}\varphi(Y)
  \right].
\end{aligned}
\end{equation}

We next compare the first and third terms on the right-hand side of
\cref{eq:2.11}. This comparison is pointwise, so fix $y\in\R^n$. Under an invertible linear change of variables, we see that
\cref{eq:2.11} is invariant, so we may first normalize $\nabla^2\varphi(y)=\Id$. In these
coordinates, $B$ is still symmetric and positive semidefinite. We may
therefore choose an orthonormal eigenbasis of $B$, which preserves this
normalization but $B$ is diagonal. In this basis we will write
\[
  \nabla^2\varphi(y)=\Id,
  \qquad
  B=\operatorname{diag}(b_1,\ldots,b_n),
  \qquad
  b_1,\ldots,b_n\geq0.
\]
The first term on the right-hand side of \cref{eq:2.11} then becomes
\[
  \varphi^{ab}(y)
  \Tr\!\left(
    \bigl(B\partial_a(\nabla^2\varphi)(y)B\bigr)\cdot
    \partial_b(\nabla^2\varphi)(y)
  \right)
  =\sum_{i,j,k=1}^n b_ib_j
  \bigl(\partial_{ijk}\varphi(y)\bigr)^2,
\]
whereas the third term on the right-hand side of \cref{eq:2.11} becomes
\[
  \left(B\nabla^2\varphi(y)B\right)_{ij}
  \varphi^{ac}(y)\varphi^{bd}(y)
  \partial_{abi}\varphi(y)\,
  \partial_{cdj}\varphi(y)
  =\sum_{i,j,k=1}^n b_i^2
  \bigl(\partial_{ijk}\varphi(y)\bigr)^2.
\]
Since the third derivatives of $\varphi$ are symmetric, interchanging
$i$ and $j$ gives
\[
  \sum_{i,j,k=1}^n b_i^2
  \bigl(\partial_{ijk}\varphi(y)\bigr)^2
  =
  \frac12\sum_{i,j,k=1}^n
  (b_i^2+b_j^2)
  \bigl(\partial_{ijk}\varphi(y)\bigr)^2.
\]
Subtracting the first term from the third therefore gives
\begin{equation}\label{eq:2.12}
\begin{aligned}
  &\left(B\nabla^2\varphi(y)B\right)_{ij}
  \varphi^{ac}(y)\varphi^{bd}(y)
  \partial_{abi}\varphi(y)\,
  \partial_{cdj}\varphi(y)\\
  &\quad-\varphi^{ab}(y)
  \Tr\!\left(
    \bigl(B\partial_a(\nabla^2\varphi)(y)B\bigr)\cdot
    \partial_b(\nabla^2\varphi)(y)
  \right)\\
  &=\frac12\sum_{i,j,k=1}^n
  (b_i^2+b_j^2)
  \bigl(\partial_{ijk}\varphi(y)\bigr)^2
  -\sum_{i,j,k=1}^n
  b_ib_j
  \bigl(\partial_{ijk}\varphi(y)\bigr)^2\\
  &=\frac12\sum_{i,j,k=1}^n
  (b_i^2+b_j^2-2b_ib_j)
  \bigl(\partial_{ijk}\varphi(y)\bigr)^2\\
  &=\frac12\sum_{i,j,k=1}^n
  (b_i-b_j)^2
  \bigl(\partial_{ijk}\varphi(y)\bigr)^2\\
  &\geq0.
\end{aligned}
\end{equation}
It remains to examine the middle term on the right-hand side of
\cref{eq:2.11}. For every $u\in\R^n$,
\[
  \left\langle
    B\nabla^2\varphi(y)Bu,u
  \right\rangle
  =\left\langle
    \nabla^2\varphi(y)Bu,Bu
  \right\rangle
  \geq0.
\]
Thus $B\nabla^2\varphi(y)B$ is positive semidefinite. Similarly,
\[
  \left\langle
    \nabla^2\varphi(y)
    \nabla^2V(\nabla\varphi(y))
    \nabla^2\varphi(y)u,u
  \right\rangle
  =\left\langle
    \nabla^2V(\nabla\varphi(y))
    \nabla^2\varphi(y)u,
    \nabla^2\varphi(y)u
  \right\rangle
  \geq0,
\]
where the final inequality follows from the convexity of $V$. The trace
of the product of these two positive semidefinite matrices is
nonnegative. Indeed, cyclicity of the trace gives
\begin{equation}\label{eq:2.13}
\begin{aligned}
  &\Tr\!\left(
    \bigl(B\nabla^2\varphi(y)B\bigr)\cdot
    \bigl(\nabla^2\varphi(y)\nabla^2V(\nabla\varphi(y))\nabla^2\varphi(y)\bigr)
  \right)\\
  &=\Tr\!\left(
    \bigl(B\nabla^2\varphi(y)B\bigr)^{1/2}
    \nabla^2\varphi(y)
    \nabla^2V(\nabla\varphi(y))
    \nabla^2\varphi(y)
    \bigl(B\nabla^2\varphi(y)B\bigr)^{1/2}
  \right)\\
  &\geq0.
\end{aligned}
\end{equation}

We may now combine the three terms in \cref{eq:2.11}. First,
\cref{eq:2.12} gives
\[
\begin{aligned}
  &\E_{Y\sim\nu}\!\left[
    \left(B\nabla^2\varphi(Y)B\right)_{ij}
    \varphi^{ac}(Y)\varphi^{bd}(Y)
    \partial_{abi}\varphi(Y)\,
    \partial_{cdj}\varphi(Y)
  \right]\\
  &\quad\geq
  \E_{Y\sim\nu}\!\left[
    \varphi^{ab}(Y)
    \Tr\!\left(
      \bigl(B\partial_a(\nabla^2\varphi)(Y)B\bigr)\cdot
      \partial_b(\nabla^2\varphi)(Y)
    \right)
  \right].
\end{aligned}
\]
The term containing $\nabla^2V$ on the right-hand side of
\cref{eq:2.11} is nonnegative by \cref{eq:2.13}.
Consequently,
\begin{equation}\label{eq:2.14}
  \E_{Y\sim\nu}\Tr\!\left(
    B\nabla^2\varphi(Y)B\nabla^2\varphi(Y)
  \right)
  \geq2\cdot\E_{Y\sim\nu}\!\left[
    \varphi^{ab}(Y)
    \Tr\!\left(
      \bigl(B\partial_a(\nabla^2\varphi)(Y)B\bigr)\cdot
      \partial_b(\nabla^2\varphi)(Y)
    \right)
  \right].
\end{equation}

We now use the Brascamp--Lieb inequality
\cite[Theorem~4.1]{BL76}, in the form of \cref{lem:a6}. Since
$B\succeq0$, its
positive semidefinite square root $B^{1/2}$ is well defined. By
\cref{lem:2.3},
\[
  \E_{Y\sim\nu}\nabla^2\varphi(Y)=\Id.
\]
Multiplying this identity on both sides by $B^{1/2}$ gives
\[
  \E_{Y\sim\nu}\!\left[
    B^{1/2}\nabla^2\varphi(Y)B^{1/2}
  \right]
  =B^{1/2}\Id B^{1/2}
  =B.
\]
For each $1\leq i,j\leq n$, set
\[
  F_{ij}(y)
  =
  \left(B^{1/2}\nabla^2\varphi(y)B^{1/2}\right)_{ij}.
\]
Consequently, for every $1\leq i,j\leq n$,
\begin{equation}\label{eq:2.15}
  \E_{Y\sim\nu}F_{ij}(Y)=B_{ij}.
\end{equation}
Again, recalling that $\nabla^2 \varphi$ is smooth and bounded, we see that each
$F_{ij}$ is smooth and bounded. Thus, for every $1\leq i,j\leq n$,
\cref{lem:a6} applied to the function $F_{ij}$ gives
\[
\begin{aligned}
  \Var_{Y\sim\nu}\!\left(F_{ij}(Y)\right)
  &=
  \E_{Y\sim\nu}\!\left[
    \bigl(F_{ij}(Y)-B_{ij}\bigr)^2
  \right]\\
  &\leq
  \E_{Y\sim\nu}\!\left[
    \varphi^{ab}(Y)
    (\partial_aF_{ij})(Y)
    (\partial_bF_{ij})(Y)
  \right].
\end{aligned}
\]
Summing these inequalities over $i$ and $j$ and using linearity of
expectation gives
\begin{equation}\label{eq:2.16}
\begin{aligned}
  \sum_{i,j=1}^n
  \Var_{Y\sim\nu}\!\left(F_{ij}(Y)\right)
  &\leq
  \sum_{i,j=1}^n
  \E_{Y\sim\nu}\!\left[
    \varphi^{ab}(Y)
    (\partial_aF_{ij})(Y)
    (\partial_bF_{ij})(Y)
  \right]\\
  &=
  \E_{Y\sim\nu}\!\left[
    \varphi^{ab}(Y)
    \sum_{i,j=1}^n
    (\partial_aF_{ij})(Y)
    (\partial_bF_{ij})(Y)
  \right].
\end{aligned}
\end{equation}
We compute the two sides of \cref{eq:2.16} separately, beginning with
the left-hand side.
By \cref{eq:2.15},
\[
  \sum_{i,j=1}^n
  \Var_{Y\sim\nu}\!\left(F_{ij}(Y)\right)
  =\E_{Y\sim\nu}
  \left\|
    B^{1/2}\nabla^2\varphi(Y)B^{1/2}-B
  \right\|_{\mathrm{HS}}^2.
\]
Since the matrix inside the Hilbert--Schmidt norm is symmetric, its
squared Hilbert--Schmidt norm is the trace of its square. Hence
\[
\begin{aligned}
  &\E_{Y\sim\nu}
  \left\|
    B^{1/2}\nabla^2\varphi(Y)B^{1/2}-B
  \right\|_{\mathrm{HS}}^2\\
  &=\E_{Y\sim\nu}
  \Tr\!\left(
    \bigl(B^{1/2}\nabla^2\varphi(Y)B^{1/2}\bigr)^2
  \right)\\
  &\quad-2\Tr\!\left(
    \E_{Y\sim\nu}\!\left[
      B^{1/2}\nabla^2\varphi(Y)B^{1/2}
    \right]B
  \right)\\
  &\quad+\Tr(B^2).
\end{aligned}
\]
Using \cref{eq:2.15} once more, the last two terms become $-2\Tr(B^2)+\Tr(B^2)=-\Tr(B^2).$ Cyclicity of the trace gives
\[
  \Tr\!\left(
    \bigl(B^{1/2}\nabla^2\varphi(Y)B^{1/2}\bigr)^2
  \right)
  =\Tr\!\left(
    B\nabla^2\varphi(Y)B\nabla^2\varphi(Y)
  \right).
\]
We conclude that the sum of the variances is
\[
  \sum_{i,j=1}^n
  \Var_{Y\sim\nu}\!\left(F_{ij}(Y)\right)
  =\E_{Y\sim\nu}
  \Tr\!\left(
    B\nabla^2\varphi(Y)B\nabla^2\varphi(Y)
  \right)
  -\Tr(B^2).
\]

We next compute the right-hand side of \cref{eq:2.16}. Since
$B^{1/2}$ is constant,
\[
  \partial_a\!\left(
    B^{1/2}\nabla^2\varphi(y)B^{1/2}
  \right)_{ij}
  =\left(
    B^{1/2}\partial_a(\nabla^2\varphi)(y)B^{1/2}
  \right)_{ij}.
\]
It follows from the definition of the Hilbert--Schmidt inner product
that
\[
  \sum_{i,j=1}^n
  \partial_a\!\left(
    B^{1/2}\nabla^2\varphi(y)B^{1/2}
  \right)_{ij}
  \partial_b\!\left(
    B^{1/2}\nabla^2\varphi(y)B^{1/2}
  \right)_{ij}
  =\Tr\!\left(
    B^{1/2}\partial_a(\nabla^2\varphi)(y)B
    \partial_b(\nabla^2\varphi)(y)B^{1/2}
  \right).
\]
Cyclicity of the trace then gives
\[
  \sum_{i,j=1}^n
  \partial_a\!\left(
    B^{1/2}\nabla^2\varphi(y)B^{1/2}
  \right)_{ij}
  \partial_b\!\left(
    B^{1/2}\nabla^2\varphi(y)B^{1/2}
  \right)_{ij}
  =\Tr\!\left(
    \bigl(B\partial_a(\nabla^2\varphi)(y)B\bigr)\cdot
    \partial_b(\nabla^2\varphi)(y)
  \right).
\]
Substituting these computations into \cref{eq:2.16} gives
\[
  \E_{Y\sim\nu}\Tr\!\left(
    B\nabla^2\varphi(Y)B\nabla^2\varphi(Y)
  \right)
  -\Tr(B^2)
  \leq\E_{Y\sim\nu}\!\left[
    \varphi^{ab}(Y)
    \Tr\!\left(
      \bigl(B\partial_a(\nabla^2\varphi)(Y)B\bigr)\cdot
      \partial_b(\nabla^2\varphi)(Y)
    \right)
  \right].
\]
By \cref{eq:2.14}, the right-hand side is at most
\[
  \frac12\cdot\E_{Y\sim\nu}
  \Tr\!\left(
    B\nabla^2\varphi(Y)B\nabla^2\varphi(Y)
  \right).
\]
Consequently, rearranging gives us
\begin{equation}\label{eq:2.17}
2\cdot \E_{Y\sim\nu}\Tr\!\left(
    B\nabla^2\varphi(Y)B\nabla^2\varphi(Y)
  \right)
  -2\Tr(B^2)
  \leq \E_{Y\sim\nu}\Tr\!\left(
    B\nabla^2\varphi(Y)B\nabla^2\varphi(Y)
  \right),
\end{equation}
which gives us the desired result whenever $B\succeq0$. It remains to remove the assumption that $B$ is positive semidefinite.
Let $B$ now be an arbitrary symmetric matrix and write
\[
  \abs{B}=(B^2)^{1/2}.
\]
Choose an orthonormal eigenbasis of $B$. In this basis,
\[
  B=\operatorname{diag}(b_1,\ldots,b_n)
\]
and
\[
  \abs{B}
  =\operatorname{diag}
  \bigl(\abs{b_1},\ldots,\abs{b_n}\bigr).
\]
At every $y\in\R^n$, let $\partial_{ij}\varphi(y)$ denote the components
of $\nabla^2\varphi(y)$ in this fixed orthonormal eigenbasis. Expanding
the trace gives
\[
  \Tr\!\left(
    B\nabla^2\varphi(y)B\nabla^2\varphi(y)
  \right)
  =\sum_{i,j=1}^nb_ib_j
  \bigl(\partial_{ij}\varphi(y)\bigr)^2.
\]
For every pair $i,j$, we simply just use $b_ib_j \leq\abs{b_i}\abs{b_j}$ and so therefore,
\[
  \Tr\!\left(
    B\nabla^2\varphi(y)B\nabla^2\varphi(y)
  \right)
  \leq\sum_{i,j=1}^n
    \abs{b_i}\abs{b_j}
    \bigl(\partial_{ij}\varphi(y)\bigr)^2
  =\Tr\!\left(
    \abs{B}\nabla^2\varphi(y)
    \abs{B}\nabla^2\varphi(y)
  \right).
\]
Taking expectations gives
\[
  \E_{Y\sim\nu}\Tr\!\left(
    B\nabla^2\varphi(Y)B\nabla^2\varphi(Y)
  \right)
  \leq\E_{Y\sim\nu}\Tr\!\left(
    \abs{B}\nabla^2\varphi(Y)
    \abs{B}\nabla^2\varphi(Y)
  \right).
\]
Since $\abs{B}\succeq0$, we may apply \cref{eq:2.17} with $\abs{B}$ in
place of $B$. We obtain
\[
  \E_{Y\sim\nu}
  \Tr\!\left(
    B\nabla^2\varphi(Y)B\nabla^2\varphi(Y)
  \right)
  \leq2\Tr(\abs{B}^2)
  =2\Tr(B^2). \qedhere
\]
\end{proof}

\subsection{Stein kernels and the \texorpdfstring{$H^{-1}$}{H-1} argument}

We now convert \cref{thm:2.5} into the desired estimate present in \cref{thm:1.2}. The object connecting these two estimates is a Stein
kernel.

\begin{definition}
Let $W\sim\lambda$ be a centered random vector in $\R^n$. A measurable
matrix valued function $\tau_\lambda:\R^n\to\R^{n\times n}$, understood up to
$\lambda$-almost-everywhere equality and whose entries are locally integrable
with respect to $\lambda$, is called a
\emph{Stein kernel} for $\lambda$ if, for every smooth compactly
supported function $f:\R^n\to\R$ and every $1\leq i\leq n$,
\begin{equation}\label{eq:2.18}
  \E_{W\sim\lambda}\!\left[W_i f(W)\right]
  =
  \E_{W\sim\lambda}\!\left[
    \sum_{j=1}^n
    \bigl(\tau_\lambda(W)\bigr)_{ij}\partial_jf(W)
  \right].
\end{equation}
If $\tau_\lambda(u)$ is symmetric for $\lambda$-almost every $u$, we call it a
symmetric Stein kernel.
\end{definition}

For background on Stein kernels and their connections with spectral gaps,
Gaussian approximation, and transport inequalities, see
\cite{LNP15,CFP19,Fathi2021}.
Stability results and transport constructions involving Stein kernels
appear in \cite{FM22,MS24}. For related
applications to Gaussian approximation, see
\cite{Mikulincer2022CLT}.
The moment map supplies the Stein kernel that we will use.

\begin{lemma}[Fathi {\cite[Theorem~2.3]{Fathi2019}}]\label{lem:2.7}
Under the regularity assumptions above for $\mu$,
for $\mu$-almost every $x$, define
\[
  \tau_\mu(x)
  =
  \nabla^2\varphi\!\left((\nabla\varphi)^{-1}(x)\right).
\]
$\tau_\mu$ is a symmetric Stein kernel for $\mu$ and is positive definite
$\mu$-almost everywhere. In the
coupling $Y\sim\nu$ and $X=\nabla\varphi(Y)\sim\mu$, one has
\[
  \tau_\mu(X)=\nabla^2\varphi(Y).
\]
\end{lemma}

\begin{proof}
By \cref{lem:2.2}, the moment map $\varphi$ is smooth and finite
on $\R^n$, and $\nabla\varphi:\R^n\to K$ is a diffeomorphism. Thus, in our notation, we see that
\[
  \tau_\mu(x)
  =
  \nabla^2\varphi\!\left((\nabla\varphi)^{-1}(x)\right),
  \qquad x\in K,
\]
is a Stein kernel for $\mu$. By \cref{lem:2.2}, this matrix is symmetric
and positive definite for every $x\in K$, hence $\mu$-almost everywhere.
Finally, if $X=\nabla\varphi(Y)$, then
\[
  \tau_\mu(X)
  =
  \nabla^2\varphi\!\left((\nabla\varphi)^{-1}
    (\nabla\varphi(Y))\right)
  =
  \nabla^2\varphi(Y).\qedhere
\]
\end{proof}

Following Barthe and Klartag \cite[Equation~3]{BK20}, for a
centered function $f\in L^2(\lambda)$ define
\begin{equation}\label{eq:2.19}
  \norm{f}_{H^{-1}(\lambda)}
  =
  \sup\left\{
    \E_{W\sim\lambda}\!\left[f(W)g(W)\right]:
    \begin{array}{l}
      g\in L^2(\lambda)\text{ is locally Lipschitz},\\
      \E_{W\sim\lambda}\abs{\nabla g(W)}^2\leq1
    \end{array}
  \right\}.
\end{equation}
We will use the following inequality of Barthe and Klartag
\cite[Proposition~10]{BK20}.

\begin{proposition}\label{prop:2.8}
Let $W\sim\lambda$ be log-concave, and let $f:\R^n\to\R$ be locally
Lipschitz. Suppose that $f,\partial_i f\in L^2(\lambda)$ and
\[
  \E_{W\sim\lambda}\partial_i f(W)=0,
  \qquad 1\leq i\leq n.
\]
Then
\begin{equation}\label{eq:2.20}
  \Var_{W\sim\lambda}\bigl(f(W)\bigr)
  \leq
  \sum_{i=1}^n
  \norm{\partial_i f}_{H^{-1}(\lambda)}^2.
\end{equation}
\end{proposition}

The measure $\lambda$ in \cref{prop:2.8} need not be isotropic and this is what allows us to use a trick for the proof of \cref{thm:1.2}, which will appear shortly. Before,
we need just a quick lemma about how a symmetric Stein kernel controls the $H^{-1}$ norm as defined in \cref{eq:2.19}.

\begin{lemma}\label{lem:2.9}
Let $W\sim\lambda$ be a centered log-concave random vector whose law is
full-dimensional and has finite second moments, and suppose that
$\tau_\lambda$ is a symmetric Stein kernel satisfying
\[
  \E_{W\sim\lambda}\norm{\tau_\lambda(W)}_{\HS}^2<\infty.
\]
Then, for every $v\in\R^n$,
\[
  \norm{x\mapsto\inner{x}{v}}_{H^{-1}(\lambda)}^2
  \leq
  \E_{W\sim\lambda}\abs{\tau_\lambda(W)v}^2.
\]
\end{lemma}

\begin{proof}
Let $g$ be a smooth compactly supported function in the test class from
\cref{eq:2.19}.
Summing the Stein identities in \cref{eq:2.18} against the coordinates
of $v$ and using the symmetry of $\tau_\lambda$ gives
\[
  \E_{W\sim\lambda}\!\left[\inner{W}{v}g(W)\right]
  =
  \E_{W\sim\lambda}
  \inner{\tau_\lambda(W)v}{\nabla g(W)}.
\]
By Cauchy--Schwarz,
\[
  \left|
    \E_{W\sim\lambda}\!\left[\inner{W}{v}g(W)\right]
  \right|
  \leq
  \left(
    \E_{W\sim\lambda}\abs{\tau_\lambda(W)v}^2
  \right)^{1/2}
  \left(
    \E_{W\sim\lambda}\abs{\nabla g(W)}^2
  \right)^{1/2}.
\]
By \cref{lem:a7}, the same estimate holds for every test
function in \cref{eq:2.19}. Taking the supremum there and then squaring
proves the lemma.
\end{proof}

Before beginning the proof, let us explain the small trick \cref{eq:2.21} that we will apply. If we apply the $H^{-1}$ inequality directly to
$x\mapsto\inner{Mx}{x}$, the Stein-kernel estimate produces
$\E_{X\sim\mu}\norm{\tau_\mu(X)M}_{\HS}^2$. This is not the matrix
expression controlled by \cref{thm:2.5}, since $M$ and
$\tau_\mu(X)$ need not commute. We therefore absorb $\abs{M}$ into the
random vector. The matrix remaining in the quadratic form is
$\operatorname{sgn}(M)$, which is orthogonal, while the squared
Hilbert--Schmidt norm of the transported Stein kernel is exactly the
quantity in \cref{thm:2.5} with $B=\abs{M}$.

\begin{proof}[Proof of \cref{thm:1.2}]
Fix a symmetric matrix $M$. By \cref{lem:a2}, it
suffices to suppose that $M$ is invertible. Recalling that we have
\[
  M=\operatorname{sgn}(M)\cdot\abs{M}
  =\abs{M}\cdot\operatorname{sgn}(M),
  \qquad
  \operatorname{sgn}(M)^T=\operatorname{sgn}(M),
  \qquad
  \operatorname{sgn}(M)^2=\Id,
\]
we thus see that $\abs{M}$ is positive definite and $\operatorname{sgn}(M)$ is
orthogonal. Set
\[
  Z=\abs{M}^{1/2}X,
  \qquad
  \eta=(\abs{M}^{1/2})_{\#}\mu.
\]
The measure $\eta$ is centered and log-concave, although it is generally
not isotropic. Its covariance and
the transformed quadratic form are
\[
  \Cov_{Z\sim\eta}(Z)=\abs{M},
  \qquad
  \inner{MX}{X}=\inner{\operatorname{sgn}(M)Z}{Z}.
\]
For $f\in C_c^\infty(\R^n)$ and $1\leq i\leq n$, the Stein identity
in \cref{eq:2.18} applied to $f\circ\abs{M}^{1/2}$, followed by the chain rule, gives
\[
\begin{aligned}
  \E[Z_i f(Z)]
  &=\sum_{k=1}^n(\abs{M}^{1/2})_{ik}
      \E[X_k f(\abs{M}^{1/2}X)]\\
  &=\E\sum_{k=1}^n\sum_{\ell=1}^n
      (\abs{M}^{1/2})_{ik}(\tau_\mu(X))_{k\ell}
      \partial_\ell(f\circ\abs{M}^{1/2})(X)\\
  &=\E\sum_{j=1}^n
      (\abs{M}^{1/2}\tau_\mu(X)\abs{M}^{1/2})_{ij}\partial_jf(Z).
\end{aligned}
\]
Thus
$\tau_\eta(z)=\abs{M}^{1/2}\tau_\mu(\abs{M}^{-1/2}z)\abs{M}^{1/2}$
is a symmetric Stein kernel
for $\eta$. In the coupling $X=\nabla\varphi(Y)$, where $Y\sim\nu$,
\cref{lem:2.7} therefore gives
\[
  \tau_\eta(Z)=\abs{M}^{1/2}\tau_\mu(X)\abs{M}^{1/2}
  =\abs{M}^{1/2}\nabla^2\varphi(Y)\abs{M}^{1/2}.
\]
The two middle $\abs{M}^{1/2}$ factors
multiply to $\abs{M}$, and cyclicity of the trace moves the last
$\abs{M}^{1/2}$ factor next to the first, producing a second copy of
$\abs{M}$. Thus
\begin{equation}\label{eq:2.21}
\begin{aligned}
  \norm{\tau_\eta(Z)}_{\HS}^2
  &=\Tr\!\left(
    \bigl(\abs{M}^{1/2}\nabla^2\varphi(Y)\abs{M}^{1/2}\bigr)^2
  \right)\\
  &=\Tr\!\left(
    \bigl(\abs{M}\nabla^2\varphi(Y)\abs{M}\bigr)
    \nabla^2\varphi(Y)
  \right).
\end{aligned}
\end{equation}
The final expression in \cref{eq:2.21} is exactly the one controlled by \cref{thm:2.5} with
$B=\abs{M}$. Consequently,
\[
\begin{aligned}
  \E_{Z\sim\eta}\norm{\tau_\eta(Z)}_{\HS}^2
  &\leq
  2\Tr(\abs{M}^2)\\
  &=
  2\Tr(M^2).
\end{aligned}
\]
In particular, \cref{lem:2.9} is satisfied. Define the centered quadratic function
\[
  F(z)=\inner{\operatorname{sgn}(M)z}{z}-\Tr(M).
\]
Since $X\sim\mu$ is isotropic, we see that $\E_{Z\sim\eta}F(Z) =
  \Tr\!\left(\operatorname{sgn}(M)\cdot\abs{M}\right)-\Tr(M)
  =0.$
Moreover,
\[
  \nabla F(z)=2\operatorname{sgn}(M)z,
  \qquad
  \E_{Z\sim\eta}\nabla F(Z)=0.
\]
Under the regularity assumptions we made above, $X$ is supported in the bounded set $K$. Therefore $Z$ is supported in
the bounded set $\abs{M}^{1/2}K$. It follows that the quadratic function
$F$ and all of its first derivatives belong to $L^2(\eta)$. Together with the two
centering identities above, this verifies all the hypotheses of
\cref{prop:2.8}. Since
\[
  \partial_iF(z)=2\inner{z}{\operatorname{sgn}(M)e_i},
\]
applying \cref{eq:2.20}, then \cref{lem:2.9}, and using the
orthogonality of $\operatorname{sgn}(M)$, we obtain
\[
\begin{aligned}
  \Var_{X\sim\mu}\bigl(\inner{MX}{X}\bigr)
  &=
  \Var_{Z\sim\eta}\bigl(F(Z)\bigr)\\
  &\leq
  4\sum_{i=1}^n
  \norm{z\mapsto\inner{z}{\operatorname{sgn}(M)e_i}}_{H^{-1}(\eta)}^2\\
  &\leq
  4\cdot\E_{Z\sim\eta}
  \sum_{i=1}^n
  \abs{\tau_\eta(Z)\cdot\operatorname{sgn}(M)e_i}^2\\
  &=
  4\cdot\E_{Z\sim\eta}
  \norm{\tau_\eta(Z)\cdot\operatorname{sgn}(M)}_{\HS}^2\\
  &=
  4\cdot\E_{Z\sim\eta}
  \norm{\tau_\eta(Z)}_{\HS}^2.
\end{aligned}
\]
We have proved that
\begin{equation}\label{eq:2.22}
  \Var_{X\sim\mu}\bigl(\inner{MX}{X}\bigr)
  \leq
  8\Tr(M^2) = 2 \cdot \E_{X\sim\mu}\abs{\nabla\inner{MX}{X}}^2
\end{equation}
whenever $M$ is invertible. By \cref{lem:a2},
\cref{eq:2.22} holds for every symmetric $M$. Finally, \cref{lem:a1} allows us to remove the regularity assumptions throughout all of \cref{sec:2} and completes the proof.
\end{proof}

\appendix

\section{Removing regularity assumptions}

We collect here the approximation arguments used in the proof.

\begin{lemma}\label{lem:a1}
Let $X \sim \mu$ be an isotropic log-concave random vector in $\R^n$. To prove \cref{thm:1.2}, it is enough to consider measures of the form
\[
  \dd\mu(x)
  =\rho(x)\,\dd x
  =e^{-V(x)}\mathbf{1}_{K}(x)\,\dd x,
\]
where $K\subset\R^n$ is bounded, open, and convex, and $V$ is smooth and
convex on an open neighborhood of $\overline K$. There are isotropic
log-concave measures $\mu_\varepsilon$ of this form such that
\[
  \E_{X\sim\mu_\varepsilon}p(X)
  \longrightarrow
  \E_{X\sim\mu}p(X)
\]
for every polynomial $p$ of degree at most four as
$\varepsilon\downarrow0$.
\end{lemma}

\begin{proof}
Let $\widetilde X_\varepsilon=X+\sqrt{\varepsilon}Z$ for an independent
standard Gaussian $Z$, and condition its law on
$\{\abs{\widetilde X_\varepsilon}<\varepsilon^{-1/2}\}$. The convolved
density is positive, smooth, and log-concave, so restricting it to this
ball and then applying an affine normalization produces an isotropic
measure $\mu_\varepsilon$ of the required form. Since $X$ has finite
eighth moment,
\[
  \E\!\left[\abs{\widetilde X_\varepsilon}^4
    \mathbf{1}_{\{\abs{\widetilde X_\varepsilon}\geq\varepsilon^{-1/2}\}}
  \right]
  \leq\varepsilon^2\E\abs{\widetilde X_\varepsilon}^8\longrightarrow0.
\]
Together with $\widetilde X_\varepsilon\to X$ in $L^4$, this shows that
the conditioned laws converge to $\mu$ in all moments through degree
four. Their means and covariances therefore converge to $0$ and $\Id$,
so the affine normalization preserves this moment convergence. Both
sides of \cref{eq:1.2} now pass to the limit.
\end{proof}

\begin{lemma}\label{lem:a2}
It is enough to prove \cref{eq:2.22} for invertible
symmetric matrices $M$.
\end{lemma}

\begin{proof}
Let $P_0$ be the orthogonal projection onto $\ker M$ and set
$M_\varepsilon=M+\varepsilon P_0$. Then $M_\varepsilon$ is invertible and
\[
  \abs{\inner{(M_\varepsilon-M)X}{X}}\leq\varepsilon\abs{X}^2,
  \qquad
  \Tr(M_\varepsilon^2)=\Tr(M^2)+\varepsilon^2\dim(\ker M).
\]
Fourth-moment integrability therefore lets
\cref{eq:2.22} pass to $M$ as $\varepsilon\downarrow0$,
since
\[
  \left(
    \E_{X\sim\mu}\abs{\inner{(M_\varepsilon-M)X}{X}}^2
  \right)^{1/2}
  \leq\varepsilon
  \left(\E_{X\sim\mu}\abs{X}^4\right)^{1/2}.\qedhere
\]
\end{proof}

\begin{lemma}\label{lem:a3}
Under the regularity assumptions of \cref{lem:a1}, fix $1\leq i\leq n$.
Suppose that $f,g:\R^n\to\R$ are smooth and that
$(\partial_i f)g$, $f\partial_i g$, $fg\partial_i\varphi$, and $fg$
belong to $L^1(\nu)$. Then
\[
  \E_{Y\sim\nu}\!\left[(\partial_i f)(Y)g(Y)\right]
  =
  \E_{Y\sim\nu}\!\left[
    f(Y)\bigl(\partial_i\varphi(Y)g(Y)-\partial_i g(Y)\bigr)
  \right].
\]
\end{lemma}

\begin{proof}
Choose $\chi\in C_c^\infty(\R^n)$ such that
$0\leq\chi\leq1$, $\chi=1$ on $B(0,1)$, and
$\operatorname{supp}\chi\subseteq B(0,2)$, and set $\chi_R(y)=\chi(y/R)$. Then
\[
  \partial_i\chi_R(y)
  =\frac1R(\partial_i\chi)(y/R),
  \qquad
  \norm{\partial_i\chi_R}_\infty
  \leq\frac{\norm{\partial_i\chi}_\infty}{R},
\]
and $\operatorname{supp}(\partial_i\chi_R)\subseteq
\{y:R\leq\abs{y}\leq2R\}$. Since
$\chi_Rfg e^{-\varphi}$ is continuously differentiable and compactly
supported, ordinary coordinate integration by parts gives
\[
\begin{aligned}
  0
  &=\int_{\R^n}
    \partial_i\!\left(\chi_R(y)f(y)g(y)e^{-\varphi(y)}\right)\,\dd y\\
  &=\E_{Y\sim\nu}\!\left[
    (\partial_i\chi_R)(Y)f(Y)g(Y)
    +\chi_R(Y)(\partial_i f)(Y)g(Y)
  \right.\\
  &\hspace{42mm}\left.
    +\chi_R(Y)f(Y)(\partial_i g)(Y)
    -\chi_R(Y)f(Y)g(Y)\partial_i\varphi(Y)
  \right].
\end{aligned}
\]
Rearranging,
\[
\begin{aligned}
  \E_{Y\sim\nu}\!\left[
    \chi_R(Y)(\partial_i f)(Y)g(Y)
  \right]
  &=\E_{Y\sim\nu}\!\left[
    \chi_R(Y)f(Y)
    \bigl(\partial_i\varphi(Y)g(Y)-\partial_i g(Y)\bigr)
  \right]\\
  &\quad-\E_{Y\sim\nu}\!\left[
    (\partial_i\chi_R)(Y)f(Y)g(Y)
  \right].
\end{aligned}
\]
The last term satisfies
\[
  \left|
    \E_{Y\sim\nu}\!\left[
      (\partial_i\chi_R)(Y)f(Y)g(Y)
    \right]
  \right|
  \leq
  \frac{\norm{\partial_i\chi}_\infty}{R}
  \E_{Y\sim\nu}\abs{f(Y)g(Y)}
  \longrightarrow0.
\]
For every fixed $y$, one has $\chi_R(y)=1$ for all sufficiently large
$R$, and $0\leq\chi_R\leq1$. The other terms therefore converge by
dominated convergence, using the assumed integrability of
$(\partial_i f)g$, $f\partial_i g$, and
$fg\partial_i\varphi$. Letting $R\to\infty$ gives the claimed identity.
\end{proof}

\begin{lemma}\label{lem:a4}
There are functions $\chi_R\in C_c^\infty(\R^n)$, defined for all
sufficiently large $R$, such that
\[
  0\leq\chi_R\leq1,
  \qquad
  \chi_R\uparrow1,
  \qquad
  \norm{\mc{L}\chi_R}_{L^1(\nu)}=O(R^{-1}).
\]
\end{lemma}

\begin{proof}
The map $\varphi$ is finite and convex, and
$\int_{\R^n}e^{-\varphi(y)}\,\dd y=1$, so it tends to infinity at infinity
and attains its minimum. Set
\[
  W=\varphi-\min_{\R^n}\varphi+1.
\]
Then
\[
  \mc{L}W
  =n-\inner{\nabla V(\nabla\varphi)}{\nabla\varphi}
\]
is bounded because $\nabla\varphi(\R^n)=K$ and both $K$ and $\nabla V$
are bounded. Write $D=\norm{\mc{L}W}_\infty$. Choose a smooth
nonincreasing function $\eta:\R\to[0,1]$ such that $\eta=1$ on
$(-\infty,1]$ and $\eta=0$ on $[2,\infty)$, and put
\[
  \chi_R=\eta(W/R).
\]
These functions are compactly supported and converge pointwise to $1$.
The chain rule gives
\[
  \mc{L}\chi_R
  =\frac1R\eta'(W/R)\mc{L}W
  +\frac1{R^2}\eta''(W/R)
  \varphi^{ij}(\partial_iW)(\partial_jW).
\]
Choose a smooth compactly supported function $w\geq\abs{\eta''}$ and
define
\[
  q_R(t)=\frac1R\int_{t/R}^{\infty}w(s)\,\dd s.
\]
Then $-q_R'(t)=R^{-2}w(t/R)$,
$\norm{q_R}_\infty\leq R^{-1}\norm{w}_{L^1}$, and $q_R(W)$ is compactly
supported. Hence \cref{eq:2.2} gives
\[
\begin{aligned}
  0
  &\leq\frac1{R^2}\E_{Y\sim\nu}\!\left[
    w(W(Y)/R)\varphi^{ij}(Y)
    (\partial_iW)(Y)(\partial_jW)(Y)
  \right]\\
  &=-\E_{Y\sim\nu}\!\left[
    \varphi^{ij}(Y)(\partial_iW)(Y)
    \partial_j(q_R\circ W)(Y)
  \right]\\
  &=\E_{Y\sim\nu}\!\left[q_R(W(Y))(\mc{L}W)(Y)\right]\\
  &\leq\frac{D\norm{w}_{L^1}}{R}.
\end{aligned}
\]
Consequently,
\[
  \norm{\mc{L}\chi_R}_{L^1(\nu)}
  \leq\frac{D}{R}
  \left(\norm{\eta'}_\infty+\norm{w}_{L^1}\right).\qedhere
\]
\end{proof}

\begin{lemma}\label{lem:a5}
Suppose that $f\in C^\infty(\R^n)\cap L^\infty(\nu)$ and that, for some
$c>0$,
\[
  \mc{L}f+cf\geq0.
\]
Then $\mc{L}f\in L^1(\nu)$ and
\[
  \E_{Y\sim\nu}\mc{L}f(Y)=0.
\]
\end{lemma}

\begin{proof}
Let $\chi_R$ be given by \cref{lem:a4}. Since $\chi_R$ is
compactly supported, symmetry in \cref{eq:2.2} gives
\[
\begin{aligned}
  \E_{Y\sim\nu}\!\left[
    \chi_R(Y)\bigl((\mc{L}f)(Y)+cf(Y)\bigr)
  \right]
  &=c\E_{Y\sim\nu}\!\left[\chi_R(Y)f(Y)\right]
    +\E_{Y\sim\nu}\!\left[f(Y)(\mc{L}\chi_R)(Y)\right].
\end{aligned}
\]
The left-hand side has a nonnegative integrand, while the right-hand
side is bounded uniformly in $R$ because $f$ is bounded and
$\norm{\mc{L}\chi_R}_{L^1(\nu)}\to0$. Fatou's lemma therefore shows that
$\mc{L}f+cf\in L^1(\nu)$, and hence $\mc{L}f\in L^1(\nu)$. Letting
$R\to\infty$ in the displayed identity gives
\[
  \E_{Y\sim\nu}\!\left[(\mc{L}f)(Y)+cf(Y)\right]
  =c\E_{Y\sim\nu}f(Y),
\]
which is the desired conclusion.
\end{proof}

\begin{lemma}\label{lem:a6}
Under the regularity assumptions of \cref{lem:a1},
every bounded smooth function $f:\R^n\to\R$ satisfies
\[
  \Var_{Y\sim\nu}(f(Y))
  \leq
  \E_{Y\sim\nu}\!\left[
    \varphi^{ab}(Y)(\partial_af)(Y)(\partial_bf)(Y)
  \right],
\]
where the right-hand side is allowed to be infinite.
\end{lemma}

\begin{proof}
The convex sets
\[
  \Omega_R=\{y\in\R^n:\varphi(y)<R\}
\]
are bounded and exhaust $\R^n$. Let
$\nu_R=\nu(\,\cdot\mid\Omega_R)$. The Brascamp--Lieb inequality
\cite[Theorem~4.1]{BL76} on $\Omega_R$ gives
\[
  \Var_{Y\sim\nu_R}(f(Y))
  \leq
  \E_{Y\sim\nu_R}\!\left[
    \varphi^{ab}(Y)(\partial_af)(Y)(\partial_bf)(Y)
  \right].
\]
Letting $R\to\infty$, boundedness handles the variance, while
$\nu(\Omega_R)\to1$ and monotone convergence handle the right-hand side.
\end{proof}

\begin{lemma}\label{lem:a7}
Let $W\sim\lambda$ be centered, full-dimensional, and log-concave with
finite second moments, and let $\tau_\lambda$ be a symmetric Stein kernel with
$\E\norm{\tau_\lambda(W)}_{\HS}^2<\infty$. For every $v\in\R^n$, the
Stein identity extends from $C_c^\infty(\R^n)$ to every locally Lipschitz
function $f:\R^n\to\R$ such that $f\in L^2(\lambda)$ and
$\nabla f\in L^2(\lambda)$:
\[
  \E\!\left[\inner{W}{v}f(W)\right]
  =\E\inner{\tau_\lambda(W)v}{\nabla f(W)}.
\]
In particular,
\[
  \left|\E\!\left[\inner{W}{v}f(W)\right]\right|
  \leq
  \left(\E\abs{\tau_\lambda(W)v}^2\right)^{1/2}
  \left(\E\abs{\nabla f(W)}^2\right)^{1/2}.
\]
\end{lemma}

\begin{proof}
By \cite[Proposition~27]{BK20}, choose
$f_k\in C_c^\infty(\R^n)$ such that
\[
  \E_{W\sim\lambda}\abs{f_k(W)-f(W)}^2\longrightarrow0,
  \qquad
  \E_{W\sim\lambda}
  \abs{\nabla f_k(W)-\nabla f(W)}^2\longrightarrow0.
\]
The two sides of the Stein identity converge by Cauchy--Schwarz, since
\[
  \E_{W\sim\lambda}\abs{\inner{W}{v}}^2<\infty,
  \qquad
  \E_{W\sim\lambda}\abs{\tau_\lambda(W)v}^2<\infty.
\]
The inequality follows from one more application of Cauchy--Schwarz.
\end{proof}

\bibliographystyle{amsplain0}
\bibliography{main}

\end{document}